\documentclass[12pt]{amsart}

\usepackage{amssymb, amsthm, amsmath}

\newtheorem{theorem}{Theorem}[section]
\newtheorem{lemma}[theorem]{Lemma}

\theoremstyle{definition}

\newtheorem{claim}{Claim}

\theoremstyle{remark}

\numberwithin{equation}{section}

\newcommand{\D}{\mathbb{D}}
\newcommand{\cD}{\overline{\D}}

\newcommand{\C}{\mathbb{C}}
\newcommand{\T}{\mathbb{T}}
\newcommand{\Z}{\mathbb{Z}}
\newcommand{\al}{\alpha}

\newcommand{\mbf}[1]{\mathbf{#1}}

\title{Schur-Agler class rational inner functions on the tridisk}
\author{Greg Knese}

\address{University of Alabama, Tuscaloosa, AL, 35487-0350}

\date{\today}

\email{geknese@bama.ua.edu}


\thanks{This research was supported by NSF grant DMS-1048775}


\begin{document}
\bibliographystyle{apalike}
\maketitle

\begin{abstract} 
We prove two results with regard to rational inner functions in the
Schur-Agler class of the tridisk.  Every rational inner function of
degree $(n,1,1)$ is in the Schur-Agler class, and every rational inner
function of degree $(n,m,1)$ is in the Schur-Agler class after
multiplication by a monomial of sufficiently high degree.
\end{abstract} 

\section{Prologue}
In this article, we continue the study of the Schur-Agler class of the
polydisk by focusing on rational inner functions.  The
\emph{Schur-Agler class} appears naturally in operator theory as the
class of holomorphic functions $f: \D^n \to \D$ which satisfy the von
Neumann inequality; i.e. for all commuting $n$-tuples of strict
contractions $(T_1,\dots,T_n)$ on some separable Hilbert space, we
have
\[
||f(T_1,\dots,T_n)|| \leq 1.
\]
The \emph{Schur class} simply refers to the holomorphic functions
$f:\D^n \to \D$.  Our general motivating question is this:
\begin{quote} How does the Schur-Agler class fit inside the Schur
  class?
\end{quote}
For $n=1,2$ these two classes coincide, but they differ for $n \geq
3$, and this is not well understood.  More recent efforts in this area
have focused on generalizations and properties of the Schur-Agler
class.  See \cite{ADR08}, \cite{BB02}, \cite{BB10}. For progress on
this question more specifically, one probably has to go back to work
of the 70's on counterexamples to von Neumann's inequality.  See
\cite{nV74}, \cite{CD75}, \cite{bL94}, \cite{jH01}.

Motivated by the recent major strides in the study of two variable
rational inner functions from \cite{CW99}, \cite{GW04}, \cite{BSV05},
along with our own efforts \cite{gK08a}, \cite{gK10c}, the approach of
this article is to make progress on this question by studying rational
inner functions in the Schur-Agler class on $\D^3$.  For further
motivation and background to this approach we refer the reader to
\cite{gK10} and \cite{gK10b}.  We now introduce our topic purely in
terms of polynomials, as our main results serve to establish a close
connection between sums of squares decompositions for positive
trigonometric polynomials and the Schur-Agler class on the tridisk
$\D^3$.

\section{Rational inner functions in the Schur-Agler class}
Let $\D, \D^n,\T, \T^n$ denote the unit disk, polydisk, torus, and
$n$-torus.  We say $p \in \C[z_1,\dots,z_n]$ has multidegree at most
$\mbf{d} = (d_1,\dots,d_n)$ if it has degree at most $d_j$ in the
variable $z_j$.

If $p$ has multidegree at most $\mbf{d}$ we may form a type of
reflection (depending on the degree)
\[
\tilde{p}(z) := z^{\mbf{d}} \overline{p(1/\bar{z}_1, \dots,
  1/\bar{z}_n)} \in \C[z_1,\dots,z_n]
\]
and if in addition $p$ has no zeros on $\D^n$, then the rational
function
\begin{equation} \label{ratinn}
\phi(z) = \frac{\tilde{p}(z)}{p(z)}
\end{equation}
is a rational inner function; i.e.\! has modulus one a.e.\! on $\T^n$
and modulus at most one on $\D^n$, by the maximum principle.  Theorem
5.2.5 of \cite{wR69} proves that every rational inner function on
$\D^n$ arises as in \eqref{ratinn}.

In particular, 
\begin{equation} \label{pz}
\begin{aligned}
|p(z)|^2 - |\tilde{p}(z)|^2 &= 0 \text{ on } \T^n \\
|p(z)|^2 - |\tilde{p}(z)|^2 &\geq 0 \text{ on } \overline{\D}^n. 
\end{aligned}
\end{equation}
On the other hand, any expression of the form
\begin{equation} \label{sos}
\sum_{j=1}^{n} (1-|z_j|^2) SOS_j(z)
\end{equation}
where each $SOS_j$ is a sum of squared moduli of polynomials, also
satisfies this inequality.  It turns out that $\tilde{p}/p$ is in the
Schur-Agler class exactly when the left side of \eqref{pz} is of the
form \eqref{sos}.

\begin{theorem} 
Given a polynomial $p \in \C[z_1,\dots,z_n]$ with no zeros in $\D^n$
and degree at most $\mbf{d}$, $\tilde{p}/p$ is in the Schur-Agler
class exactly when there exists a decomposition
\[
|p(z)|^2 - |\tilde{p}(z)|^2 = \sum_{j=1}^{n} (1-|z_j|^2) SOS_j(z)
\]
where each $SOS_j$ is a sum of squared moduli of polynomials.
\end{theorem}

This theorem is implicit in \cite{CW99}.  To take a trivial example,
set $p(z)= 1$ which we momentarily view as having degree at most
$(1,1,\dots,1)$.  Then, a decomposition would be
\[
\begin{aligned}
&1-|z_1\cdots z_n|^2 \\
=& (1-|z_1|^2)+(1-|z_2|^2)|z_1|^2\\
& +
(1-|z_3|^2)|z_1z_2|^2 + \cdots + (1-|z_n|^2)|z_1\cdots z_{n-1}|^2.
\end{aligned}
\]
Here the sums of squares terms are each a single square.  For a
non-trivial example see \cite{gK10}.

We refined the above theorem as follows.

\begin{theorem}[\cite{gK10}] \label{bounds} 
If $p \in \C[z_1,\dots, z_n]$ has multi-degree at most $\mbf{d} =
(d_1,\dots, d_n)$ and $\tilde{p}/p$ is in the Schur-Agler class, then
given a decomposition:
\[
|p(z)|^2 - |\tilde{p}(z)|^2 = \sum_{j=1}^{n} (1-|z_j|^2) K_j(z,z)
\]
where each $K_j$ is a positive semi-definite function, it must be
the case that $K_j$ is a sum of squares of polynomials of degree at
most
\[
\begin{cases} d_j -1 \text{ in } z_j &\\
d_k \text{ in } z_k  &\text{ for } k \ne j
\end{cases}
\]
 In particular, $K_j$ can be written as a sum of at most
 $d_j\prod_{k\ne j}(d_k+1)$ polynomials (by dimensionality).
\end{theorem}

Recall that a function $K(z,\zeta)$ is positive semi-definite if for
every finite set $F$ the matrix
\[
(K(z,\zeta))_{z,\zeta \in F}
\]
is positive semi-definite.  (We would need an ordering to form an
actual matrix, but this is unimportant.)  For more information on
positive semi-definite kernels, refer to \cite{AM02} Section 2.7.

The main results of this paper relate to rational inner functions in
the Schur-Agler class on $\D^3$.  The first interesting result in this
area is due to Kummert.

\begin{theorem}[\cite{aK89b}] \label{kummert} 
If $p\in \C[z_1,z_2,z_3]$ has degree $(1,1,1)$ and has no zeros on
$\overline{\D}^3$, then $\tilde{p}/p$ is in the Schur-Agler class.
\end{theorem}

We gave the following minor improvement to the details of the sums of
squares decomposition of $\tilde{p}/p$ in \cite{gK10b}.

\begin{theorem} \label{sharpkummert} 
If $p \in \C[z_1,z_2,z_3]$ has degree $(1,1,1)$ and no zeros on
$\overline{\D}^3$, then there exist sums of squares terms such that
\[
|p|^2 - |\tilde{p}|^2 = \sum_{j=1}^{3} (1-|z_j|^2)SOS_j(z)
\]
where $SOS_3$ is a sum of two squares, while $SOS_1$, $SOS_2$ are sums
of four squares.
\end{theorem}

Our two main results are the following.  We improve the above results
to the case of polynomials of degree $(n,m,1)$ and exhibit a new
phenomenon in the study of the Schur-Agler class.

\begin{theorem} \label{thm1}
If $p\in \C[z_1,z_2,z_3]$ has degree $(n,1,1)$ and no zeros on
$\overline{\D}^3$, then $\tilde{p}/p$ is in the Schur-Agler class.
Furthermore, we have a decomposition
\[
|p|^2 - |\tilde{p}|^2 = \sum_{j=1}^{3} (1-|z_j|^2)SOS_j(z)
\]
where $SOS_3$ is a sum of two squares, while $SOS_1$, $SOS_2$ are sums
of $4(n-1), 2(n+1)$ squares respectively.
\end{theorem}

\begin{theorem} \label{thm2}
If $p \in \C[z_1,z_2,z_3]$ has no zeros on $\overline{\D}^3$ and
degree at most $(n,m,1)$, then there exist integers $r,s \geq 0$ such
that
\[
\frac{z_1^r z_2^s \tilde{p}(z_1,z_2,z_3)}{p(z_1,z_2,z_3)}
\]
is in the Schur-Agler class.  
\end{theorem}

This phenomenon has not been observed in the study of the Schur-Agler
class (although it is analogous to results in ``sums of squares'' such
as Quillen's theorem \cite{dQ}).  We do not yet have an example of $p$
such that $\tilde{p}/p$ is not Schur-Agler while $z_1^r z_2^s
\tilde{p}/p$ is.  However, in the last section we explain how a
construction might go.

We proceed to two necessary preliminary results and then to the proof
of Theorems \ref{thm1} and \ref{thm2}.

\section{Preliminary results}
The following result is proven in \cite{aM03}, \cite{mD04},
\cite{GL06}, and \cite{bD07}.  

\begin{theorem} \label{trigsos}
Suppose $t$ is a $d$ variable, positive trigonometric polynomial:
\[
t(z) = \sum_{-N \leq |\al| \leq N} t_{\al} z^\al > 0 \text{ for all } z
= (z_1,\dots, z_d) \in \T^d
\]
where we use multi-index notation with $\al \in \Z^d$.  Then $t$ can
be written as a sum of squares of polynomials; i.e. there exist $A_j
\in \C[z_1,\dots, z_d]$ such that
\[
t(z) = \sum_{j=1}^{M} |A_j(z)|^2 \qquad (z \in \T^n).
\]
\end{theorem}

Known proofs of this result require both strict positivity and can
only control the numbers of polynomials (and their degrees) in the
sums of squares decomposition in terms of a bound below on $t$.  See
\cite{GL06} for more detail.  It is this subtlety that creates the
need to multiply by a sufficiently high degree monomial in Theorem
\ref{thm2}.  We get around this in Theorem \ref{thm1} via the
following lemma.

\begin{lemma} 
Let $t$ be a non-negative, two variable trigonometric polynomial of
degree $(n,1)$, i.e.
\begin{equation} \label{teq}
t(z_1,z_2) = t_0(z_1) + z_2t_1(z_1) + \overline{z_2t_1(z_1)} \geq 0 
\end{equation}
where $t_0, t_1$ are one variable trigonometric polynomials of degree
at most $n$.  Then, there exist $A_1,A_2 \in \C[z_1,z_2]$ of degree at
most $(n,1)$ such that
\[
t(z) = |A_1(z)|^2 + |A_2(z)|^2 \qquad (z \in \T^2).
\]
\end{lemma}

This lemma is implicitly known, but in a different language/context
(and with more complicated proofs and less detail).  In \cite{jG98}
and \cite{BG98}, it is proven and phrased in the language of
trigonometric moment problems.  The connection to sums of squares is
because of the main result of \cite{wR63}, which, loosely speaking,
says that given a subset $\Lambda$ of $\Z_+^d$, one can solve
(truncated) trigonometric moment problems on $\Lambda-\Lambda$ if and
only if non-negative trigonometric polynomials with Fourier support in
$\Lambda-\Lambda$ are sums of squares of polynomials with coefficient
support in $\Lambda$.
 
\begin{proof}
The proof is really the same as the degree $(1,1)$ case, which we gave
in \cite{gK10b}.  By minimizing \eqref{teq} over $z_2$, we see that
$t_0(z_1) \geq 2|t_1(z_1)|$.  This implies that the $2\times 2$ matrix
trigonometric polynomial
\[
T(z_1) = \begin{bmatrix} \frac{1}{2} t_0(z_1) & t_1(z_1) \\
\overline{t_1(z_1)} & \frac{1}{2} t_0(z_1) \end{bmatrix}
\]
is positive semi-definite.  By the matrix Fej\'er-Riesz theorem (due
to M.Rosenblum, see \cite{mD04} for a recent proof and references),
$T$ can be factored as
\[
T(z_1) = A(z_1)^* A(z_1)
\]
where $A \in \C^{2\times 2}[z_1]$ is a matrix polynomial of degree at
most $n$.  Then, 
\[
t(z_1,z_2) = \begin{bmatrix} 1 & \bar{z}_2 \end{bmatrix}
T(z_1) \begin{bmatrix} 1 \\ z_2 \end{bmatrix} = \left\|
A(z_1) \begin{bmatrix} 1 \\ z_2 \end{bmatrix} \right\|^2
\]
is a sum of two squares of the desired type.
\end{proof}

\section{Proof of Theorems \ref{thm1} and \ref{thm2}}

To prove Theorems \ref{thm1} and \ref{thm2} simultaneously we merely
need to keep track of whether we are using the lemma or Theorem
\ref{trigsos}.  A brief notational warning: if $E$ is a column vector
of polynomials in the variables $z_1,z_2$ (as will occur below), we
shall write $\|E(z_1,z_2)\|^2$ for the sum of squares of the entries
of $E$, and often to save space we write $\|E\|^2$ for the same
expression.  These are all pointwise euclidean norms and do not
represent any kind of function space norm.

Write $p(z) = a(z_1,z_2) + b(z_1,z_2) z_3$ where $a,b \in \C[z_1,z_2]$
have degree at most $(n,m)$.  For $z_1,z_2 \in \T$, by direct
computation
\begin{equation} \label{lurkisom1}
|p|^2 - |\tilde{p}|^2 = (1-|z_3|^2)(|a(z_1,z_2)|^2- |b(z_1,z_2)|^2).
\end{equation}
Then, for $(z_1,z_2) \in \T^2$, $|a(z_1,z_2)|^2- |b(z_1,z_2)|^2$ is a
non-negative two variable trig polynomial of degree at most $(n,m)$.
As $p$ has no zeros on $\overline{\D}^3$, $|a|^2 - |b|^2$ is in fact
strictly positive on $\T^2$, since a zero $(z'_1,z'_2)$ would imply
$|p(z'_1,z'_2, \cdot)| = |\tilde{p}(z'_1,z'_2,\cdot)|$ and this would
mean $z_3 \mapsto p(z'_1,z'_2,z_3)$ has a zero on $\T$.

By the lemma or by Theorem \ref{trigsos}, we may write
\[
|a(z_1,z_2)|^2- |b(z_1,z_2)|^2 = \|E(z_1,z_2)\|^2 \text{ on } \T^2
\]
where $E$ is a vector polynomial with values in $\C^N$ (this provides
a convenient way to represent sums of squares).  In the degree
$(n,1,1)$ case we may take $N=2$ and in the $(n,m,1)$ case we do not
know what $N$ is.  Set $\tilde{E}(z_1,z_2) = z_1^{n+r} z_2^{m+s}
\overline{E(1/\bar{z}_1, 1/\bar{z}_2)}$, where we assume $E$ has
degree $(n+r,m+s)$.  Again, in the case $m=1$, we may take $r=s=0$.

We also remark that since $p$ has no zeros on $\overline{\D}^3$, $a$
has no zeros on $\overline{\D}^2$.  By the maximum principle 
\[
\frac{\tilde{b}(z_1,z_2)}{a(z_1,z_2)}
\]
is analytic and has modulus strictly less than one since $|b| =
|\tilde{b}|$ on $\T^2$ and since $|a| > |b|$ on $\T^2$. In
particular, $a+z_1^r z_2^s \tilde{b}$ has no zeros on
$\overline{\D}^2$.

We may polarize formula \eqref{lurkisom1} and get for $z_1,z_2 \in \T$
\begin{equation} \label{lurkisom2}
p(z_1,z_2,z_3) \overline{p(z_1,z_2,\zeta_3)} - \tilde{p}(z_1,z_2,z_3)
\overline{\tilde{p} (z_1,z_2,\zeta_3)} = (1-z_3 \bar{\zeta}_3)
\|E(z_1,z_2)\|^2,
\end{equation}
for $z_3,\zeta_3 \in \C$, which we rearrange into
\[
\begin{aligned}
& p(z_1,z_2,z_3) \overline{p(z_1,z_2,\zeta_3)} + z_3 \bar{\zeta}_3
\|E(z_1,z_2)\|^2 \\
&= \tilde{p}(z_1,z_2,z_3)
\overline{\tilde{p} (z_1,z_2,\zeta_3)} + \|E(z_1,z_2)\|^2.
\end{aligned}
\]

Then, for fixed $z_1,z_2 \in \T$ and for varying $z_3$, the map
\begin{equation} \label{lurkisom33}
\begin{bmatrix} p(z_1,z_2,z_3) \\ z_3 E(z_1,z_2) \end{bmatrix}
\mapsto \begin{bmatrix} z_1^r z_2^s\tilde{p}(z_1,z_2,z_3)
  \\ E(z_1,z_2) \end{bmatrix}
\end{equation}
gives a well-defined isometry $V(z_1,z_2)$ (which depends on
$z_1,z_2$) from the span of the elements on the left to the span of
the elements on the right (the span taken over the above vectors as
$z_3$ varies).  
More concretely, by examining coefficients of $z_3$, we map
\begin{equation} \label{lurkisom4}
\begin{bmatrix} a(z_1,z_2) \\ \vec{0} \end{bmatrix}
\mapsto \begin{bmatrix} z_1^r z_2^s\tilde{b}(z_1,z_2)
  \\ E(z_1,z_2) \end{bmatrix}, \qquad \begin{bmatrix} b(z_1,z_2)
  \\ E(z_1,z_2) \end{bmatrix} \mapsto \begin{bmatrix} z_1^r z_2^s
  \tilde{a}(z_1,z_2) \\ \vec{0} \end{bmatrix}.
\end{equation}

This is how the ``lurking isometry argument'' traditionally works,
however $V(z_1,z_2)$ does not extend uniquely to define a unitary on
$\C^{N+1}$ and we would like to extend $V(z_1,z_2)$ so that $V$ is
rational in $z_1,z_2$.

The definition that Kummert gives in the $(1,1,1)$ case works here
with a small modification.

\begin{claim} 
Define
\[
V = \frac{1}{a} \begin{bmatrix} z_1^r z_2^s
  \tilde{b} & \tilde{E}^t \\ E &
  \frac{E\tilde{E}^t - a(z_1^r
    z_2^s\tilde{a}+b)I}{a+z_1^r
    z_2^s\tilde{b}}.
\end{bmatrix}
\]
Then, $V$ is holomorphic in $\D^2$ and unitary valued on $\T^2$; 
$V$ satisfies 
\begin{equation} \label{lurkisom3}
V(z_1,z_2) \begin{bmatrix} p(z_1,z_2,z_3) \\ z_3
  E(z_1,z_2) \end{bmatrix} = \begin{bmatrix} z_1^r
  z_2^s\tilde{p}(z_1,z_2,z_3) \\ E(z_1,z_2) \end{bmatrix}
\end{equation}
 for $(z_1,z_2) \in \T^2$ and hence for all $(z_1,z_2) \in \cD^2$ by
 analyticity.  (This is just the content of \eqref{lurkisom33}.)
\end{claim}

\begin{proof}[Proof of Claim]
First, $V$ is holomorphic since $a$ and $a+ z_1^r z_2^s\tilde{b}$ have
no zeros on $\cD^2$.

Let 
\[
S(z_1,z_2) = \text{span}\left\{\begin{bmatrix} p(z_1,z_2,z_3) \\ z_3
  E(z_1,z_2) \end{bmatrix}: z_3 \in \C\right\}
\]
and notice that 
\[
S(z_1,z_2) = \text{span}\left\{\begin{bmatrix} z_1^r
z_2^s\tilde{p}(z_1,z_2,z_3) \\ E(z_1,z_2) \end{bmatrix}: z_3 \in
\C\right\}.
\]
The goal is to show $V(z_1,z_2)$ is a unitary by verifying
\eqref{lurkisom3}, which shows $V(z_1,z_2)$ is isometric on the
subspace $S(z_1,z_2)$, and by showing $V(z_1,z_2)$ maps
$S(z_1,z_2)^{\perp}$ isometrically into itself.

To show \eqref{lurkisom3} we first observe that
\[
\begin{aligned}
\tilde{E}(z_1,z_2)^tE(z_1,z_2) &= z_1^{n+r}z_2^{m+s}\|E(z_1,z_2)\|^2 \\ &=
z_1^{n+r}z_2^{m+s}(|a(z_1,z_2)|^2 - |b(z_1,z_2)|^2) = z_1^{r} z_2^{s}( a
\tilde{a} - b\tilde{b}).
\end{aligned}
\]
where $a,b$ are reflected at degree $(n,m)$ to give $\tilde{a},
\tilde{b}$.

Here are the computations used to show \eqref{lurkisom4} (which is
equivalent to \eqref{lurkisom3}):
\[
V \begin{bmatrix} a \\ \vec{0} \end{bmatrix}
= \begin{bmatrix} z_1^r z_2^r \tilde{b} \\ E \end{bmatrix}
\]
and 
\[
V  \begin{bmatrix} b
  \\ E \end{bmatrix} = \frac{1}{a} \begin{bmatrix} z_1^r z_2^s \tilde{b} b +
z_1^rz_2^s(a\tilde{a} - b\tilde{b}) \\ \frac{ b(a+z_1^rz_2^s\tilde{b}) +
z_1^rz_2^s(a\tilde{a} - b\tilde{b}) - a(z_1^rz_2^s \tilde{a} + b)}{a+
  z_1^r z_2^s \tilde{b}} E \end{bmatrix}
 = \begin{bmatrix} z_1^r z_2^s
  \tilde{a} \\ \vec{0} \end{bmatrix}
\]

Now we show $V(z_1,z_2)$, viewed as a linear map, is isometric on the
orthogonal complement of $S(z_1,z_2)$.  Set $X(z_1,z_2)$ equal to the
orthogonal complement of $S(z_1,z_2)$ in $\C^{1+N}$, and observe that
\[
X(z_1,z_2) = \{ \begin{bmatrix} 0 \\ v \end{bmatrix}: v \perp E(z_1,z_2) \}.
\]
Notice that $v \perp E(z_1,z_2)$ if $\tilde{E}^t v = 0$. 

Let us observe what the definition of $V$ does to elements of $X$.
For $\vec{x} = \begin{bmatrix} 0 \\ v \end{bmatrix} \in X$, we have
\[
V\vec{x} = \begin{bmatrix} 0 \\ -\frac{z_1^r z_2^s \tilde{a} + b}{a +
    z_1^r z_2^s \tilde{b}} v \end{bmatrix} = -\frac{z_1^r z_2^s
  \tilde{a} + b}{a + z_1^r z_2^s \tilde{b}} \vec{x}.
\]
So, every element of $X$ is an eigenvector with eigenvalue
$-\frac{z_1^r z_2^s\tilde{a}+b}{a + z_1^rz_2^s\tilde{b}}$.  This
number is unimodular for $(z_1,z_2) \in \T^2$.  This proves
$V(z_1,z_2)$ is unitary valued and the claim is proved.
\end{proof}

This means $V$ is an $(N+1)\times (N+1)$ two-variable rational
matrix-valued inner function.  It was proved in \cite{aK89} (see also
\cite{BSV05}) that such functions have transfer function
representations.  Namely, there exists a $((N+1) + n_1 + n_2) \times
((N+1)+n_1+n_2)$ block unitary
\[
U = \begin{bmatrix} A & B \\ C & D \end{bmatrix} = \begin{bmatrix} A &
  B_1 & B_2 \\
C_1 & D_{11} & D_{12} \\
C_2 & D_{21} & D_{22} \end{bmatrix}
\]
where $A$ is an $(N+1)\times (N+1)$ matrix, $B$ is an $(N+1)\times
(n_1+n_2)$, $C$ is an $(n_1+n_2) \times (N+1)$, $D$ is an $(n_1+n_2)
\times (n_1+n_2)$ (all subdivided as indicated) such that 
\begin{equation} \label{transfer}
V(z_1,z_2)
= A + B d(z_1,z_2) (I - D d(z_1,z_2))^{-1} C
\end{equation}
 where
\[
d(z_1,z_2) = \begin{bmatrix} z_1I_1 & 0 \\ 0 & z_2 I_2 \end{bmatrix}.
\]
Here $I_1, I_2$ are the $n_1$, $n_2$-dimensional identity matrices,
respectively. 

Such a representation is equivalent to the formula
\begin{equation} \label{lurkisom5}
U \begin{bmatrix} I \\ z_1 G_1(z_1,z_2) \\ z_2 G_2(z_1,z_2) \end{bmatrix}
= \begin{bmatrix} V(z_1,z_2) \\ G_1(z_1,z_2) \\ G_2(z_1,z_2) \end{bmatrix}
\end{equation}
where $G_1, G_2$ are $\C^{n_1\times(N+1)}$, $\C^{n_2\times(N+1)}$ valued functions given by
\[
\begin{bmatrix} G_1(z_1,z_2) \\ G_2(z_1,z_2) \end{bmatrix} = (I-D
d(z_1,z_2))^{-1} C.  
\]
Indeed, one can use formula \eqref{lurkisom5} to explicitly solve for
$V$ as in \eqref{transfer}.

Define 
\[
Y = \begin{bmatrix} p \\ z_3 E\end{bmatrix} \text{ and } H_j = G_j Y
  \text{ for } j =1,2.
\]
Then, by these definitions
\[
U \begin{bmatrix} I \\ z_1 G_1 \\ z_2 G_2 \end{bmatrix} Y =
U \begin{bmatrix} Y \\ z_1 G_1 Y \\ z_2 G_2 Y \end{bmatrix} =
U \begin{bmatrix} p \\ z_3 E \\ z_1 H_1 \\ z_2 H_2 \end{bmatrix}
\]
and 
\[
\begin{bmatrix} V \\  G_1 \\  G_2 \end{bmatrix} Y = \begin{bmatrix} V
  Y \\ H_1 \\ H_2 \end{bmatrix} = \begin{bmatrix} z_1^r z_2^s\tilde{p}
  \\ E \\ H_1 \\ H_2 \end{bmatrix}
\]
where the last equality follows from \eqref{lurkisom3}.

By \eqref{lurkisom5} we now have
\[
U \begin{bmatrix} p \\ z_3 E \\ z_1 H_1 \\ z_2 H_2 \end{bmatrix}
= \begin{bmatrix} z_1^r z_2^s \tilde{p} \\ E \\ H_1 \\ H_2 \end{bmatrix}
\]
and since $U$ is a unitary we have
\[
\begin{aligned}
& |p|^2 + |z_3|^2\|E\|^2 + |z_1|^2\|H_1\|^2 + |z_2|^2\|H_2\|^2 \\
& = |z_1^r z_2^s \tilde{p}|^2 + \|E\|^2 + \|H_1\|^2 + \|H_2\|^2
\end{aligned}
\]
which can be rearranged to give
\[
|p|^2 - |z_1^r z_2^s \tilde{p}|^2 = \sum_{j=1,2} (1-|z_j|^2) \|H_j\|^2 +
(1-|z_3|^2) \|E\|^2
\]
Even though we have not verified that $H_1$ and $H_2$ are polynomials,
this is enough to prove $z_1^rz_2^s\tilde{p}/p$ is in the Schur-Agler
class.  In fact, Theorem \ref{bounds} forces $\|H_1\|^2$, $\|H_2\|^2$
to be sums of squares of polynomials of multi-degree $(n+r-1,m+s,1)$,
$(n+r,m+s-1,1)$.

In the case $m=1$, we have $r=s=0$, $E$ of degree $(n,1,0)$, and $E$
has values in $\C^2$. So, $\|E\|^2$ is a sum of two squares, and by
dimensional considerations $\|H_1\|^2$ is a sum of at most $4(n-1)$
squares and $\|H_2\|^2$ is a sum of at most $2(n+1)$ squares.
This concludes the proof of Theorems \ref{thm1} and \ref{thm2}.

\section{How to construct examples}
How might one construct an example of a rational inner function
$\tilde{p}/p$ with $p$ of degree $(n,m,1)$ and the property that
$\tilde{p}/p$ is not in the Schur-Agler class while $z_1^rz_2^s
\tilde{p}/p$ is? Examining the above proof, it is a matter of choosing
$a, b \in \C[z_1,z_2]$ such that (1) $a$ has no zeros on $\D^2$, (2)
$|a|^2-|b|^2\geq 0$ on $\T^2$, and (3) $|a|^2-|b|^2$ is not a sum of
squares of polynomials of degree $(n,m)$.  Positive trigonometric
polynomials of degree $(n,m)$ which cannot be written as a sum of
squares of polynomials of degree $(n,m)$ do exist (see \cite{mD04} and
\cite{bD07}), so our problem reduces to finding such a trigonometric
polynomial of the form $|a|^2-|b|^2$.  We leave this for future work
(or future authors).

\bibliography{sacrifott}

\end{document}